\documentclass[reqno]{amsart}

\usepackage[T1]{fontenc}
\usepackage{amssymb}
\usepackage{enumitem}
\usepackage{mathrsfs}
\usepackage[colorlinks, linkcolor=blue, citecolor=blue, urlcolor=blue]{hyperref}
\usepackage{tikz}

\newtheorem{theorem}{Theorem}[section]
\newtheorem{corollary}[theorem]{Corollary}
\newtheorem{lemma}[theorem]{Lemma}

\theoremstyle{definition}

\numberwithin{equation}{section}

\newcommand{\nsum}{\mathbin{\#}}
\DeclareMathOperator{\ran}{ran}
\DeclareMathOperator{\rk}{rk}
\DeclareMathOperator{\hgt}{ht}
\DeclareMathOperator{\ot}{ot}
\DeclareMathOperator{\cf}{cf}
\DeclareMathOperator{\Gr}{Gr}

\title{A choice-free proof of the Erd\H{o}s--Dushnik--Miller theorem}

\author{Guozhen Shen}
\address{Department of Philosophy (Zhuhai)\\
Sun Yat-sen University\\
No.~2 Daxue Road\\
Zhuhai\\
Guangdong Province 519082\\
People's Republic of China}
\email{shen\_guozhen@outlook.com}

\thanks{The author used ChatGPT (OpenAI) for assistance with language editing and for discussing and checking mathematical arguments during the preparation of this manuscript. All mathematical statements, arguments, and proofs included in the final version were independently verified by the author, who takes full responsibility for the content of the article. The author was partially supported by National Natural Science Foundation of China grant number 12671003.}



\begin{document}

\begin{abstract}
The Erd\H{o}s--Dushnik--Miller theorem states that for every aleph $\kappa$,
\[
\kappa\to(\kappa,\omega);
\]
that is, every coloring $c:[\kappa]^2\to2$ has either a $0$-homogeneous set of cardinality $\kappa$
or a $1$-homogeneous set of cardinality $\omega$.
In this article, we present a purely combinatorial proof of this theorem in $\mathsf{ZF}$
(i.e., Zermelo--Fraenkel set theory without the axiom of choice), avoiding any metamathematical considerations.
\end{abstract}

\maketitle

\section{Introduction}
In 1941, Dushnik and Miller~\cite{Dushnik1941} proved (using the axiom of choice) the following theorem:
if $G$ is a graph of cardinality $\kappa$, where $\kappa$ is an aleph,
and if every subset of $G$ of cardinality $\kappa$ contains two connected elements,
then $G$ contains a complete graph of cardinality $\omega$.
They credited Erd\H{o}s with the proof of the result when $\kappa$ is a singular cardinal,
and the result is now generally known as the Erd\H{o}s--Dushnik--Miller theorem.
Using the Erd\H{o}s--Rado arrow notation, the result can be stated~as
\[
\kappa\to(\kappa,\omega);
\]
that is, every coloring $c:[\kappa]^2\to2$ has either a $0$-homogeneous set of cardinality $\kappa$
or a $1$-homogeneous set of cardinality $\omega$.

In an unpublished manuscript~\cite{Karagila2014}, Karagila used absoluteness to show that
the Erd\H{o}s--Dushnik--Miller theorem can be proved in $\mathsf{ZF}$.
Recently, Csern\'ak and Soukup~\cite{Csernak2025} restated Karagila's ideas and asked for
an elementary (combinatorial) proof for the Erd\H{o}s--Dushnik--Miller theorem (see~\cite[Problem~4.3]{Csernak2025}).
In this article, we present such a proof, avoiding any metamathematical considerations.

\section{Preliminaries}
Throughout this article, we shall work in $\mathsf{ZF}$. In this section, we briefly introduce
the terminology and notation we use and present some results needed to prove our main theorems.

\subsection{Well-founded posets and trees}
A \emph{partially ordered set} (or \emph{poset}) is a relational structure
$\langle P,\prec\rangle$ such that $\prec$ is irreflexive and transitive.
A poset $\langle P,\prec\rangle$ is \emph{well-founded} if every nonempty subset of $P$ has a $\prec$-minimal element.

For a well-founded poset $\langle P,\prec\rangle$, we define the \emph{rank} function recursively by
\[
\rk_\prec(q)=\textstyle\sup^+\{\rk_\prec(p)\mid p\prec q\},
\]
where $\sup^+$ denotes the least strict upper bound. The \emph{height} of $\langle P,\prec\rangle$ is defined by
\[
\hgt(P)=\textstyle\sup^+\{\rk_\prec(p)\mid p\in P\}.
\]
An easy induction shows that, for all $q\in P$,
\[
\rk_\prec(q)=\{\rk_\prec(p)\mid p\prec q\},
\]
and hence
\[
\hgt(P)=\{\rk_\prec(p)\mid p\in P\}.
\]
For every $\alpha<\hgt(P)$, the \emph{$\alpha$-th level} of $\langle P,\prec\rangle$ is defined by
\[
P_\alpha=\{p\in P\mid\rk_\prec(p)=\alpha\}.
\]
A \emph{chain} in $\langle P,\prec\rangle$ is a subset $C\subseteq P$ that is totally ordered by $\prec$;
that is, for all $p,q\in C$, either $p\prec q$, $p=q$, or $q\prec p$.

A \emph{tree} is a well-founded poset $\langle T,\prec\rangle$ such that $\{p\in T\mid p\prec q\}$ is a chain for every $q\in T$.
A maximal chain in a tree is called a \emph{branch}.

\subsection{Order types and cofinalities}
Let $\langle P,\prec\rangle$ be a well-ordered set (i.e., a well-founded, totally ordered set).
Then there is a unique ordinal, called the \emph{order type} of $\langle P,\prec\rangle$ and denoted by $\ot(P)$,
and there is a unique isomorphism, denoted by $e_P$, from $P$ onto $\ot(P)$.

Let $\kappa$ be an aleph (i.e., an infinite well-ordered cardinal).
Clearly, for all $A\subseteq\kappa$, $\ot(A)<\kappa$ if and only if $|A|<\kappa$.
It is well known that the set $\kappa^{<\omega}$ of all finite sequences on $\kappa$ has cardinality $\kappa$.
The cofinality of $\kappa$ is denoted by $\cf(\kappa)$.
For every $\alpha<\cf(\kappa)$, every function from $\alpha$ to $\kappa$ is bounded.

\subsection{Natural sums}
For ordinals $\alpha$ and $\beta$, we can uniquely represent them as
\begin{align*}
\alpha & =\omega^{\gamma_0}\cdot k_0+\dots+\omega^{\gamma_{n-1}}\cdot k_{n-1},\\
\beta  & =\omega^{\gamma_0}\cdot l_0+\dots+\omega^{\gamma_{n-1}}\cdot l_{n-1},
\end{align*}
where $\gamma_0>\dots>\gamma_{n-1}$ and $0<k_i+l_i<\omega$ for every $i<n$.
The \emph{natural sum} $\alpha\nsum\beta$ of $\alpha$ and $\beta$ is defined by
\[
\alpha\nsum\beta=\omega^{\gamma_0}\cdot(k_0+l_0)+\dots+\omega^{\gamma_{n-1}}\cdot(k_{n-1}+l_{n-1}).
\]
Clearly, for every aleph $\kappa$ and every $\alpha,\beta<\kappa$, we have $\alpha\nsum\beta<\kappa$.

\begin{lemma}\label{S001}
For any sets $A$ and $B$ of ordinals, $\ot(A\cup B)\leqslant\ot(A)\nsum\ot(B)$.
\end{lemma}
\begin{proof}
See~\cite[IV.2.22(vii)]{Levy1979}.
\end{proof}

\section{Kurepa's theorem}
In 1935, Kurepa proved (using the axiom of choice) the following theorem (see~\cite[Proposition~7.9]{Kanamori2003} or \cite[IX.2.32]{Levy1979}).
Let $\kappa$ and $\lambda$ be alephs such that $\lambda<\cf(\kappa)$.
If $\langle T,\prec\rangle$ is a tree of height $\kappa$ such that $|T_\alpha|<\lambda$ for all $\alpha<\kappa$,
then $\langle T,\prec\rangle$ has a branch of order type $\kappa$.

In this section, we prove, without using the axiom of choice,
the case $\lambda=\omega$ of Kurepa's theorem for arbitrary well-founded posets.
To do this, we first prove the following lemma.

We say that a well-founded poset $\langle P,\prec\rangle$ is \emph{gapless}
if for every $p,q\in P$ with $p\prec q$ and every ordinal $\gamma$ with $\rk_\prec(p)<\gamma<\rk_\prec(q)$,
there exists $r\in P$ such that $p\prec r\prec q$ and $\rk_\prec(r)=\gamma$.

\begin{lemma}\label{S002}
Let $\langle P,\prec\rangle$ be a well-founded poset each of whose levels is finite.
Then there exists ${\prec_\omega}\subseteq{\prec}$ such that $\langle P,\prec_\omega\rangle$
is a gapless well-founded poset and $\rk_\prec(q)=\rk_{\prec_\omega}(q)$ for all $q\in P$.
\end{lemma}
\begin{proof}
We define a relation $\prec'$ on $P$ by
\[
p\prec'q\iff p\prec q\wedge\forall\gamma(\rk_\prec(p)<\gamma<\rk_\prec(q)\rightarrow\exists r(p\prec r\prec q\wedge\rk_\prec(r)=\gamma)).
\]
Clearly, $\langle P,\prec'\rangle$ is a well-founded poset. We claim that, for all $q\in P$,
\begin{equation}\label{S003}
\rk_\prec(q)=\rk_{\prec'}(q).
\end{equation}

We prove \eqref{S003} by $\prec$-induction on $q$. Assume inductively that $\rk_\prec(p)=\rk_{\prec'}(p)$ for every $p\prec q$.
We have to show that $\rk_\prec(q)=\rk_{\prec'}(q)$. Since ${\prec'}\subseteq{\prec}$,
it follows from the induction hypothesis that $\rk_\prec(q)\geqslant\rk_{\prec'}(q)$.

We prove $\rk_\prec(q)\leqslant\rk_{\prec'}(q)$ as follows. If $\rk_\prec(q)=0$, this is obvious.
If $\rk_\prec(q)$ is a successor ordinal, say $\beta+1$, then $\beta=\rk_\prec(p)=\rk_{\prec'}(p)$ for some $p\prec q$.
Clearly, $p\prec'q$, and hence $\rk_\prec(q)=\beta+1\leqslant\rk_{\prec'}(q)$. Suppose now that $\rk_\prec(q)$ is a limit ordinal.
Let $\alpha<\rk_\prec(q)$. By the induction hypothesis, it suffices to prove that $\alpha=\rk_\prec(p)$ for some $p\prec'q$.
Assume to the contrary that $p\nprec'q$ for all $p\prec q$ with $\rk_\prec(p)=\alpha$.
By definition, for each $p\prec q$ with $\rk_\prec(p)=\alpha$,
there is a least ordinal $\gamma_p$ such that $\alpha<\gamma_p<\rk_\prec(q)$
and there are no $r\in P$ with $p\prec r\prec q$ and $\rk_\prec(r)=\gamma_p$.
Let
\[
\delta=\textstyle\sup^+\{\gamma_p\mid p\prec q\wedge\rk_\prec(p)=\alpha\}.
\]
Since each level of $\langle P,\prec\rangle$ is finite and $\rk_\prec(q)$ is a limit ordinal,
it follows that $\delta<\rk_\prec(q)$, and hence $\delta=\rk_\prec(s)$ for some $s\prec q$.
By the induction hypothesis, $\alpha<\delta=\rk_{\prec'}(s)$,
and hence $\alpha=\rk_{\prec'}(p)=\rk_\prec(p)$ for some $p\prec's$.
Since $p\prec's$ and $\alpha<\gamma_p<\delta$, it follows that $\gamma_p=\rk_\prec(r)$ for some $r\in P$
with $p\prec r\prec s$ and $\rk_\prec(r)=\gamma_p$, which is a contradiction. Thus \eqref{S003} is proved.

We define by recursion
\begin{align*}
  {\prec_0}      & ={\prec}, \\
  {\prec_{n+1}}  & ={\prec_n'}, \\
  {\prec_\omega} & =\bigcap_{n\in\omega}{\prec_n}.
\end{align*}
By \eqref{S003}, an easy induction shows that, for every $n\in\omega$,
$\langle P,\prec_n\rangle$ is a well-founded poset and $\rk_\prec(q)=\rk_{\prec_n}(q)$ for all $q\in P$.
It is then easy to see that $\langle P,\prec_\omega\rangle$ is a well-founded poset.

We prove that $\rk_\prec(q)=\rk_{\prec_\omega}(q)$ by $\prec$-induction on $q$.
Assume inductively that $\rk_\prec(p)=\rk_{\prec_\omega}(p)$ for every $p\prec q$.
Since ${\prec_\omega}\subseteq{\prec}$, it follows from the induction hypothesis that $\rk_\prec(q)\geqslant\rk_{\prec_\omega}(q)$.
To prove the other direction, let $\alpha<\rk_\prec(q)$.
For each $n\in\omega$, since $\alpha<\rk_\prec(q)=\rk_{\prec_n}(q)$,
there exists $p\prec_nq$ such that $\rk_\prec(p)=\rk_{\prec_n}(p)=\alpha$.
Since each level of $\langle P,\prec\rangle$ is finite, there exists $p\in P$ with $\rk_\prec(p)=\alpha$
such that $p\prec_nq$ for infinitely many $n\in\omega$, and hence $p\prec_\omega q$.
By the induction hypothesis, $\rk_{\prec_\omega}(p)=\rk_\prec(p)=\alpha$. Therefore, $\alpha<\rk_{\prec_\omega}(q)$.

Finally, we prove that $\langle P,\prec_\omega\rangle$ is gapless.
Let $p,q\in P$ with $p\prec_\omega q$ and let $\gamma$ be an ordinal such that $\rk_{\prec_\omega}(p)<\gamma<\rk_{\prec_\omega}(q)$.
For each $n\in\omega$, since $p\prec_{n+1}q$, ${\prec_{n+1}}={\prec_n'}$, and $\rk_{\prec_n}(p)<\gamma<\rk_{\prec_n}(q)$,
it follows that there exists $r\in P$ such that $p\prec_nr\prec_nq$ and $\rk_\prec(r)=\rk_{\prec_n}(r)=\gamma$.
Since each level of $\langle P,\prec\rangle$ is finite, there exists $r\in P$ with $\rk_\prec(r)=\gamma$
such that $p\prec_nr\prec_nq$ for infinitely many $n\in\omega$, and hence $p\prec_\omega r\prec_\omega q$ and $\rk_{\prec_\omega}(r)=\gamma$.
\end{proof}

For a well-founded poset $\langle P,\prec\rangle$, we say that an element $p\in P$ is a \emph{blind alley} if
\[
\textstyle\sup^+\{\rk_\prec(q)\mid p\preccurlyeq q\}<\hgt(P).
\]
Clearly, if $\langle P,\prec\rangle$ is gapless, then for every $p\in P$ that is not a blind alley and every ordinal
$\gamma$ with $\rk_\prec(p)<\gamma<\hgt(P)$, there exists $r\in P_\gamma$ such that $p\prec r$.

The main idea behind the proof of the following theorem is taken from~\cite{Caicedo2010}.

\begin{theorem}\label{S004}
Let $\kappa$ be an aleph and let $\langle P,\prec\rangle$ be a well-founded poset each of whose levels is finite.
If $|P|=\kappa$, then $\langle P,\prec\rangle$ has a chain of order type $\kappa$.
\end{theorem}
\begin{proof}
By Lemma~\ref{S002}, we may assume without loss of generality that $\langle P,\prec\rangle$ is gapless.
Fix a well-ordering $<$ of $P$, which exists since $|P|=\kappa$. We claim that $\hgt(P)\geqslant\kappa$.
Assume to the contrary that $\hgt(P)<\kappa$. If $\kappa=\omega$, then $\hgt(P)<\omega$,
so $P$ is a finite union of finite levels and is therefore finite, which is a contradiction.
If $\kappa>\omega$, then the function that maps each $p\in P$ to $\langle\rk_\prec(p),e_{P_{\rk_\prec(p)}}(p)\rangle$
is an injection from $P$ into $\hgt(P)\times\omega$. Thus $\kappa=|P|\leqslant|{\hgt(P)}|\cdot\omega<\kappa$, which is also a contradiction.
Hence, $\hgt(P)\geqslant\kappa$.

Let $Q$ be the set of all elements of $P$ that are not blind alleys.
Clearly, $Q$ is a downset in $P$; that is, for every $p,q\in P$, if $p\prec q$ and $q\in Q$, then $p\in Q$.
Hence, $\langle Q,\prec\rangle$ is a gapless well-founded poset, and its rank function agrees with that of $\langle P,\prec\rangle$ on $Q$.
We claim that $\langle Q,\prec\rangle$ has the same height as $\langle P,\prec\rangle$ and contains no blind alleys.

Assume towards a contradiction that $\alpha=\hgt(Q)<\hgt(P)$.
Then $P_\alpha\cap Q=\varnothing$; that is, every element of $P_\alpha$ is a blind alley.
Since $P_\alpha$ is finite, it follows that
\[
\beta=\textstyle\max\bigl\{\sup^+\{\rk_\prec(q)\mid p\preccurlyeq q\}\bigm|p\in P_\alpha\bigr\}<\hgt(P).
\]
Let $q\in P_\beta$. Since $\alpha<\beta$, there exists $p\in P_\alpha$ such that $p\prec q$.
Then $\beta=\rk_\prec(q)<\beta$, which is a contradiction. Hence, $\hgt(Q)=\hgt(P)\geqslant\kappa$.

Assume towards a contradiction that $\langle Q,\prec\rangle$ contains a blind alley $p$;
that is, $\gamma=\sup^+\{\rk_\prec(q)\mid p\preccurlyeq q\wedge q\in Q\}<\hgt(Q)=\hgt(P)$.
Since $p$ is not a blind alley in~$P$, $\{r\in P_\gamma\mid p\prec r\}\neq\varnothing$.
Since $P_\gamma$ is finite and all elements of $\{r\in P_\gamma\mid p\prec r\}$ are blind alleys in $P$, it follows that
\[
\delta=\textstyle\max\bigl\{\sup^+\{\rk_\prec(s)\mid r\preccurlyeq s\}\bigm|r\in P_\gamma\wedge p\prec r\bigr\}<\hgt(P).
\]
Let $s\in P_\delta$ with $p\prec s$. Since $\langle P,\prec\rangle$ is gapless,
there exists $r\in P_\gamma$ such that $p\prec r\prec s$. Then $\delta=\rk_\prec(s)<\delta$, which is a contradiction.

Now we recursively construct a chain in $\langle Q,\prec\rangle$ of order type $\kappa$.
Let $x_0$ be the $<$-least element of $Q_0$. If $\alpha<\kappa$ and $x_\alpha\in Q_\alpha$ has been defined,
let $x_{\alpha+1}$ be the $<$-least element of $Q_{\alpha+1}$ such that $x_\alpha\prec x_{\alpha+1}$.
Such an element exists because $x_\alpha$ is not a blind alley in $Q$.
Let $\beta<\kappa$ be a limit ordinal, and assume that the chain $\{x_\alpha\mid\alpha<\beta\}$
has been constructed so that $x_\alpha\in Q_\alpha$ for every $\alpha<\beta$.
Since $\langle Q,\prec\rangle$ contains no blind alleys, for every $\alpha<\beta$ there exists $r\in Q_\beta$ such that $x_\alpha\prec r$.
Since $Q_\beta$ is finite, there must exist an $r\in Q_\beta$ such that $x_\alpha\prec r$ for cofinally many $\alpha<\beta$.
Let $x_\beta$ be the $<$-least such $r$. Then $x_\beta\in Q_\beta$ and $x_\alpha\prec x_\beta$ for all $\alpha<\beta$.

Hence, $\langle Q,\prec\rangle$, and therefore $\langle P,\prec\rangle$, has a chain of order type $\kappa$.
\end{proof}

The following corollary is essential for the proof of the main theorem.

\begin{corollary}\label{S005}
Let $\kappa$ be an aleph, let $P\subseteq\kappa$ with $|P|=\kappa$, and let $f:P\to\kappa$.
Then there exists $Q\subseteq P$ with $|Q|=\kappa$ on which $f$ is monotonic;
that is, for all $\alpha,\beta\in Q$, if $\alpha<\beta$, then $f(\alpha)\leqslant f(\beta)$.
\end{corollary}
\begin{proof}
We define a relation $\prec$ on $P$ by
\[
\alpha\prec\beta\iff\alpha<\beta\wedge f(\alpha)\leqslant f(\beta).
\]
Clearly, $\langle P,\prec\rangle$ is a well-founded poset. Every level of $\langle P,\prec\rangle$ is finite,
since, whenever $\alpha$ and $\beta$ belong to the same level and $\alpha<\beta$, we have $f(\alpha)>f(\beta)$.
By Theorem~\ref{S004}, $\langle P,\prec\rangle$ has a chain $Q$ of order type $\kappa$.
Then $|Q|=\kappa$ and $f$ is monotonic on $Q$.
\end{proof}

\section{Blocker trees}
Let $\kappa$ be an aleph and let $c:[\kappa]^2\to2$ be a coloring.
For every $X\subseteq\kappa$, the \emph{greedy subset} $\Gr(X)\subseteq X$ is defined recursively as follows:
\[
\beta\in\Gr(X)\iff\forall\alpha<\beta(\alpha\in\Gr(X)\rightarrow c(\{\alpha,\beta\})=0).
\]
Clearly, $\Gr(X)$ is $0$-homogeneous, and $\Gr(X)\neq\varnothing$ whenever $X\neq\varnothing$.
By definition, for each $\beta\in X\setminus\Gr(X)$, $\{\alpha\in\beta\cap\Gr(X)\mid c(\{\alpha,\beta\})=1\}\neq\varnothing$.

Let $Q\subseteq\kappa$ with $|Q|=\kappa$. The \emph{blocker tree} $T_Q$ is defined level by level as follows.
The nodes of $T_Q$ are finite sequences on $Q$, partially ordered by $\subset$,
and each node $s\in T_Q$ carries a nonempty set $X_s\subseteq Q$ and its greedy subset $Y_s=\Gr(X_s)$.
At the root, define $X_\varnothing=Q$. Assume that $s\in T_Q$ and that $X_s\subseteq Q$ has been defined.
For each $\beta\in X_s\setminus Y_s$, let
\[
a_s(\beta)=\min\{\alpha\in\beta\cap Y_s\mid c(\{\alpha,\beta\})=1\}.
\]
For $\alpha\in Y_s$, retain the child $s{}^\frown\langle\alpha\rangle$ precisely when
\[
X_{s{}^\frown\langle\alpha\rangle}=\{\beta\in X_s\setminus Y_s\mid a_s(\beta)=\alpha\}\neq\varnothing.
\]
Note that $X_{s{}^\frown\langle\alpha\rangle}\subseteq X_s\setminus Y_s\subset X_s$ and
$X_{s{}^\frown\langle\alpha\rangle}\cap X_{s{}^\frown\langle\alpha'\rangle}=\varnothing$ whenever $\alpha\neq\alpha'$.
Note also that $Y_s=\Gr(X_s)$ is $0$-homogeneous for every $s\in T_Q$.

It is easy to see that, for all $s,t\in T_Q$, if $s\subset t$, then $X_s\supset X_t$,
and if $s$ and $t$ are incomparable, then $X_s\cap X_t=\varnothing$.
Hence $\{X_s\mid s\in T_Q\}$, partially ordered by $\supset$, is a tree isomorphic to $T_Q$.
As a consequence, for all distinct $s,t\in T_Q$, $Y_s\cap Y_t=\varnothing$.

\begin{lemma}\label{S006}
For every $t\in T_Q$, $t$ is increasing and $\ran(t)$ is $1$-homogeneous.
\end{lemma}
\begin{proof}
Let $\alpha,\beta\in\ran(t)$ be such that $\alpha$ precedes $\beta$ in $t$,
and let $s$ and $s'$ be the initial segments of $t$ preceding $\alpha$ and $\beta$, respectively. Then
\[
\beta\in Y_{s'}\subseteq X_{s'}\subseteq X_{s{}^\frown\langle\alpha\rangle},
\]
which implies that $\alpha=a_s(\beta)<\beta$ and $c(\{\alpha,\beta\})=1$.
\end{proof}

\begin{lemma}\label{S007}
If $c$ has no $1$-homogeneous sets of cardinality $\omega$, then
\[
\sup\{\ot(Y_s)\mid s\in T_Q\}=\kappa,
\]
where $\ot(Y_s)$ is the order type of $Y_s$ with respect to $<$.
\end{lemma}
\begin{proof}
Suppose that $c$ has no $1$-homogeneous sets of cardinality $\omega$. We claim that
\begin{equation}\label{S008}
\text{$\{Y_s\mid s\in T_Q\}$ is a partition of $Q$.}
\end{equation}
Clearly, $\{Y_s\mid s\in T_Q\}$ is a family of pairwise disjoint nonempty subsets of $Q$.
It remains to show that $Q\subseteq\bigcup\{Y_s\mid s\in T_Q\}$.

Assume towards a contradiction that there exists $\beta\in Q\setminus\bigcup\{Y_s\mid s\in T_Q\}$.
Recursively define $s_0=\varnothing$ and $s_{n+1}=s_n{}^\frown\langle a_{s_n}(\beta)\rangle$ for each $n\in\omega$.
An easy induction shows that $s_n\in T_Q$ and $\beta\in X_{s_n}\setminus Y_{s_n}$ for all $n\in\omega$.
By Lemma~\ref{S006}, $\bigcup_{n\in\omega}\ran(s_n)$ is a $1$-homogeneous set of cardinality $\omega$,
contradicting our assumption. Thus \eqref{S008} is proved.

We proceed to prove the lemma. Assume towards a contradiction that
\[
\lambda=\sup\{\ot(Y_s)\mid s\in T_Q\}<\kappa.
\]
If $\kappa=\omega$, then $\lambda<\omega$, and hence $Y_s$ is finite for all $s\in T_Q$.
An easy induction shows that every level of $T_Q$ is finite.
Since $T_Q\subseteq Q^{<\omega}$, it follows that $|T_Q|\leqslant|Q^{<\omega}|=\omega$.
By \eqref{S008}, $T_Q$ is infinite, and hence $|T_Q|=\omega$.
By Theorem~\ref{S004}, $T_Q$ has a chain $C$ of order type $\omega$.
Hence, by Lemma~\ref{S006}, $\bigcup_{s\in C}\ran(s)$ is a $1$-homogeneous set of cardinality $\omega$,
contradicting our assumption.

Suppose now that $\kappa>\omega$. Let $\beta\in Q$. By \eqref{S008}, there exists a unique $t\in T_Q$ such that $\beta\in Y_t$,
say $t=\langle\alpha_0,\dots,\alpha_{n-1}\rangle$. We encode $\beta$ by the finite sequence
\[
\pi(\beta)=\langle\varepsilon_0,\dots,\varepsilon_{n-1},\varepsilon_n\rangle,
\]
where $\varepsilon_i=e_{Y_{t{\upharpoonright i}}}(\alpha_i)$ for $i<n$ and $\varepsilon_n=e_{Y_t}(\beta)$.
The coding $\pi$ is injective, because for every $i<n$,
$\alpha_i$ can be recovered recursively from $\langle\alpha_0,\dots,\alpha_{i-1}\rangle$ and $\varepsilon_i$,
and $\beta$ can be recovered from $t$ and $\varepsilon_n$.
Since $\ot(Y_s)\leqslant\lambda$ for all $s\in T_Q$, it follows that $\pi$ is an injection from $Q$ into $\lambda^{<\omega}$.
Hence, $\kappa=|Q|\leqslant|\lambda^{<\omega}|\leqslant\max\{|\lambda|,\omega\}<\kappa$, which is a contradiction.
\end{proof}

\section{The Erd\H{o}s--Dushnik--Miller theorem}
Now we are ready to prove the main theorem.

\begin{theorem}\label{S009}
For every aleph $\kappa$,
\[
\kappa\to(\kappa,\omega);
\]
that is, every coloring $c:[\kappa]^2\to2$ has either a $0$-homogeneous set of cardinality $\kappa$
or a $1$-homogeneous set of cardinality $\omega$.
\end{theorem}
\begin{proof}
Let $\kappa$ be an aleph and let $c:[\kappa]^2\to2$ be a coloring.
Suppose that $c$ has no $1$-homogeneous sets of cardinality $\omega$.
Our goal is to construct a $0$-homogeneous set of cardinality $\kappa$.

For $A\subseteq\kappa$ and $\alpha\in A$, write
\[
N_A(\alpha)=\{\beta\in A\setminus\{\alpha\}\mid c(\{\alpha,\beta\})=1\}
\]
for the \emph{total color-$1$ neighborhood} of $\alpha$ in $A$.
We claim that there exists $P\subseteq\kappa$ with $|P|=\kappa$ such that $|N_P(\alpha)|<\kappa$ for all $\alpha\in P$.
Assume towards a contradiction that for every $A\subseteq\kappa$ with $|A|=\kappa$
there exists $\alpha\in A$ such that $|N_A(\alpha)|=\kappa$.
Recursively define $A_0=\kappa$, $\alpha_n=\min\{\alpha\in A_n\mid|N_{A_n}(\alpha)|=\kappa\}$,
and $A_{n+1}=N_{A_n}(\alpha_n)$ for each $n\in\omega$. Then $\{\alpha_n\mid n\in\omega\}$
is a $1$-homogeneous set of cardinality $\omega$, contradicting our assumption.

Let $f:P\to\kappa$ be defined by
\[
f(\alpha)=\ot(N_P(\alpha)).
\]
By Corollary~\ref{S005}, there exists $Q\subseteq P$ with $|Q|=\kappa$ on which $f$ is monotonic.
Let $T_Q$ be the blocker tree of $Q$.

Let $\langle\gamma_\xi\rangle_{\xi<\cf(\kappa)}$ be a cofinal sequence in $\kappa$.
We recursively define bounded subsets $D_\xi\subseteq Q$ for $\xi<\cf(\kappa)$ as follows.
Assume that $D_\eta$ has been defined for all $\eta<\xi$.
For each $\eta<\xi$, let $\delta_\eta$ be the least element of $Q$ that is an upper bound of $D_\eta$,
and let $\theta_\xi$ be the least element of $Q$ such that $\delta_\eta<\theta_\xi$ for all $\eta<\xi$.
Then $D_\eta\subseteq\theta_\xi$ for every $\eta<\xi$.
Since $f$ is monotonic on $Q$, for every $\eta<\xi$ and every $\alpha\in D_\eta$, we have
\[
\ot(N_P(\alpha))=f(\alpha)\leqslant f(\delta_\eta)\leqslant f(\theta_\xi).
\]
Let
\[
B_\xi=\bigcup_{\eta<\xi}\bigcup_{\alpha\in D_\eta}N_P(\alpha).
\]
For each $\beta\in B_\xi$, let $\alpha_\beta$ be the least element of $\bigcup_{\eta<\xi}D_{\eta}$
such that $\beta\in N_P(\alpha_\beta)$. Then the map
\[
\beta\longmapsto\langle\alpha_\beta,e_{N_P(\alpha_\beta)}(\beta)\rangle
\]
is an injection from $B_\xi$ into $\theta_\xi\times f(\theta_\xi)$. Hence
\[
|B_\xi|\leqslant|\theta_\xi|\cdot|f(\theta_\xi)|<\kappa.
\]
Let $\beta_\xi=\ot(B_\xi)<\kappa$. Then $\beta_\xi\nsum\gamma_\xi<\kappa$.
Note that $T_Q\subseteq Q^{<\omega}\subseteq\kappa^{<\omega}$.
Fix the length-first lexicographic well-ordering of $\kappa^{<\omega}$.
By Lemma~\ref{S007}, there is a least $s_\xi\in T_Q$ such that
\[
\ot(Y_{s_\xi})>\beta_\xi\nsum\gamma_\xi.
\]
Hence $\ot(Y_{s_\xi}\setminus B_\xi)>\gamma_\xi$, since otherwise, by Lemma~\ref{S001},
\[
\ot(Y_{s_\xi})\leqslant\ot(B_\xi)\nsum\ot(Y_{s_\xi}\setminus B_\xi)\leqslant\beta_\xi\nsum\gamma_\xi,
\]
which is a contradiction. Now we define $D_\xi$ to be the first $\gamma_\xi$ elements of $Y_{s_\xi}\setminus B_\xi$.
Then $D_\xi$ is a bounded subset of $Q$ of order type $\gamma_\xi$.

Finally, let
\[
H=\bigcup_{\xi<\cf(\kappa)}D_\xi.
\]
For every $\xi<\cf(\kappa)$, we have $\ot(H)\geqslant\ot(D_\xi)=\gamma_\xi$. Thus $|H|=\kappa$.
We conclude the proof by showing that $H$ is $0$-homogeneous.
Let $\alpha,\beta$ be distinct elements of $H$. If $\alpha,\beta\in D_\xi$ for some $\xi<\cf(\kappa)$,
then $c(\{\alpha,\beta\})=0$ since $D_\xi\subseteq Y_{s_\xi}$ is $0$-homogeneous.
Otherwise, by symmetry, suppose that $\alpha\in D_\eta$ and $\beta\in D_\xi$ for some $\eta<\xi<\cf(\kappa)$.
Then $\beta\notin B_\xi$, and hence $\beta\notin N_P(\alpha)$, which implies that $c(\{\alpha,\beta\})=0$.
\end{proof}

We conclude the article with the following corollary of Theorem~\ref{S009}, which is also mentioned in~\cite{Caicedo2010}.

\begin{corollary}
If $P$ is infinite and $<$ and $\prec$ are two well-orderings of $P$,
then there exists a subset $Q\subseteq P$ with $|Q|=|P|$ on which $<$ and $\prec$ coincide.
\end{corollary}


\begin{thebibliography}{99}

\bibitem{Caicedo2010} A.~Caicedo,
Distinct well-orderings of the same set,
MathOverflow (2010), \url{https://mathoverflow.net/q/40507}.

\bibitem{Csernak2025} T.~Csern\'ak and L.~Soukup,
Infinite combinatorics revisited in the absence of Axiom of choice,
\emph{Arch. Math. Logic} \textbf{64} (2025), 473--491.

\bibitem{Dushnik1941} B.~Dushnik and E.~W.~Miller,
Partially ordered sets,
\emph{Amer. J. Math.} \textbf{63} (1941), 600--610.

\bibitem{Kanamori2003} A.~Kanamori,
\emph{The Higher Infinite: Large Cardinals in Set Theory from their Beginnings},
2nd ed., Springer Monogr. Math., Springer, Berlin, 2003.

\bibitem{Karagila2014} A.~Karagila,
Absolutely Choiceless Proofs,
arXiv preprint (2014), arXiv:1402.3048.

\bibitem{Levy1979} A.~Levy,
\emph{Basic Set Theory},
Springer, Berlin, 1979.

\end{thebibliography}
\end{document}